\documentclass[12pt,oneside]{amsart}
\usepackage{amsmath,amssymb,latexsym,soul,cite,mathrsfs}
\usepackage{enumitem}
\usepackage{comment}
\usepackage{tikz}
\usetikzlibrary{decorations.pathreplacing}
\usepackage{lmodern}
\usepackage{ragged2e} 
\usepackage{color,enumitem,graphicx}
\usepackage[colorlinks=true,urlcolor=blue,
citecolor=red,linkcolor=blue,linktocpage,pdfpagelabels,
bookmarksnumbered,bookmarksopen]{hyperref}
\usepackage[english]{babel}
\usepackage{lineno}
\usepackage[left=2.4cm,right=2.4cm,top=2.5cm,bottom=2.4cm]{geometry}

\usepackage[hyperpageref]{backref}
\numberwithin{equation}{section}
\newtheorem{theorem}{Theorem}
\newtheorem{lemma}[theorem]{Lemma}

\title[Maximizers for the singular Trudinger--Moser functional]{Maximizers for the Singular Trudinger--Moser functional beyond the critical regime}

\author[J.F.\ de Oliveira]{Jos\'{e} Francisco de Oliveira}\thanks{This work was partially supported by the National Council for Scientific and Technological Development (CNPq): the first author (grant Nos. 309491/2021-5 and 303443/2025-1), the second author (grant No. 308395/2023-9), and the third author (grant No. 310885/2023-0).}
\author[M. de Souza]{Manass\'{e}s de Souza}
\author[E. Medeiros]{Everaldo Medeiros}

\address[J.F.\ de Oliveira]{
\newline\indent Department of Mathematics
\newline\indent 
	Federal University of Piau\'{i}
\newline\indent
	64049-550, Teresina, PI, Brazil}
	\email{\href{mailto:jfoliveira@ufpi.edu.br}{jfoliveira@ufpi.edu.br}}
	
	\address[M. de Souza]{
		\newline\indent Department of Mathematics
		\newline\indent Federal University of Para\'\i ba
		\newline\indent 58051-900, Jo\~ao Pessoa, PB, Brazil}
	\email{\href{mailto:manasses.xavier@academico.ufpb.br}{manasses.xavier@academico.ufpb.br}}
	
	\address[E. Medeiros]{
			\newline\indent Department of Mathematics
		\newline\indent Federal University of Para\'\i ba
		\newline\indent 58051-900, Jo\~ao Pessoa, PB, Brazil}
	\email{\href{mailto:everaldo@mat.ufpb.br}{everaldo@mat.ufpb.br}}
	

\begin{document}
\begin{abstract}
Our aim is to investigate the existence of local maximizers for the singular Trudinger-Moser functional 
$$
F_{\alpha}(u) := \int_{\Omega} \frac{e^{\alpha u^2}-1}{|x|^{a}} \mathrm{d}x,\;\;\; u \in W_0^{1,2}(\Omega),
$$ 
restricted to the manifold
$
\Sigma=:\left\{u\in W_0^{1,2}(\Omega):\|\nabla u\|_2=1\right\},
$
where $a\in [0, 2)$, $\alpha\ge 0$ and $\Omega$ denotes a smooth bounded domain  $\mathbb{R}^2$ containing the origin.   Adimurthi and Sandeep (Nonlinear.
Differ. Equ. Appl. \textbf{13}, 2007) showed the following singular Trudinger-Moser type estimate   
\begin{equation}\nonumber
    \sup_{u \in W_0^{1,2}(\Omega), \, \|\nabla u\|_{L^2} \le 1} F_\alpha(u)<\infty\;\;\;\mbox{iff}\;\;\alpha \leq \alpha_a:=2\pi(2-a).
\end{equation}
In particular,  the functional $F_{\alpha}$ is bounded on $ \Sigma$ whenever $\alpha \le \alpha_a$. In addition,   Csat\'{o} and Roy (Calc. Var. Partial Differ. Equ., \textbf{54},  2015)  were able to ensure the existence of maximizers for $F_{\alpha}$ on $\Sigma$ when $\alpha\le \alpha_a$. In the  supercritical regime $\alpha > \alpha_a$, the functional $F_{\alpha}$ becomes unbounded on $\Sigma$. Nevertheless, we prove that $F_{\alpha}$ still possesses local maximizers on $\Sigma$ beyond the critical threshold, at least for $\alpha > \alpha_a$ sufficiently close to $\alpha_a$. Our approach relies on a variational analysis near the set of maximizers associated with the critical parameter $\alpha_a$, together with a suitable local compactness argument. Our result improves and complements related findings due to Struwe (Ann. Inst. Henri Poincar\'{e}, Anal. Non Lin\'{e}aire, \textbf{5}, 1988).
\end{abstract}

\maketitle

\section{Introduction and main result} 
\noindent For a smooth bounded domain $\Omega\subset \mathbb{R}^2$, it follows from the  result of J.~Moser  \cite{Moser} (see also  \cite{Yudovich1961, Pohozaev1965, Peetre,Trudinger1967}) that
\begin{equation}\label{I-Moser}
  \sup_{u \in W_0^{1,2}(\Omega), \, \|u\|=1}
\int_{\Omega}\big(e^{\alpha u^{2}} - 1\big)\, \mathrm{d}x<\infty \quad\mbox{iff}\;\; \alpha\le 4\pi,
\end{equation}
where  $W_0^{1,2}(\Omega)$ is the Sobolev space endowed with the Dirichlet norm
$
\|u\| = \left(\int_\Omega |\nabla u|^2 \, \mathrm{d}x\right)^{{1}/{2}}.
$ 
Estimates such as \eqref{I-Moser} are currently known as Trudinger-Moser type inequalities and are fundamental to the study of nonlinear differential equations in geometric analysis and physics, see for instance  \cite{Chang87, Baird, Moser}. Their importance has led to extensive research on generalizations and variants; for a broad overview, we refer to \cite{CR,Adams, LamLu2,AS} and the references therein.

A central problem in geometric analysis involves studying the critical points of the functional 
\begin{equation}\label{F-Moser}
J(u) = \int_{\Omega} (e^{u^2} - 1) \mathrm{d}x
\end{equation}
constrained to the manifold $M_\alpha = \{u \in W_0^{1,2}(\Omega) : \|u\|^2 = \alpha\}$. For $\alpha\le 4\pi$,  the attainability of \eqref{I-Moser} proved by Flucher \cite{Flucher} (see also \cite{Carleson-Chang,Struwe,Lin1996}) ensures a critical point for $J_{|_{M_{\alpha}}}$. When $\Omega=B$ is the unit ball in $\mathbb{R}^2$, although $J$ is unbounded in $M_{\alpha}$ for $\alpha>4\pi$,   Malchiodi and Martinazzi \cite{MM} were able to show that there exists $\alpha^{\sharp}>4\pi$  for which  $J_{|_{M_\alpha}}$ has at least two positive critical points for $\alpha \in (4\pi, \alpha^{\sharp})$, but no positive critical points for $\alpha > \alpha^{\sharp}$. This result improves a previous one due to Struwe \cite{Struwe} and to Lamm, Robert, and Struwe \cite{LRS}. Recently, da Silva-Paiva and de Oliveira explored the existence of critical points beyond the critical threshold for the logarithmic weight Trudinger-Moser functional, see \cite{CR,Ruan}.  For a general domain $\Omega \subset \mathbb{R}^2$, it was shown in \cite{PMR} that the topology of $\Omega$ has a strong influence on the existence of critical points for $J_{|_{M{\alpha}}}$.

A significant extension of the Trudinger-Moser inequality \eqref{I-Moser} for singular weights was proved by Adimurthi and Sandeep \cite{AS}. Precisely, let $\Omega\subset\mathbb{R}^2$ be a bounded domain containing the origin. Then, for every $a \in [0, 2)$ and $\alpha>0$, the singular Moser's functional 
\begin{equation}\label{SF-Moser}
F_{\alpha}(u) = \int_{\Omega} \frac{e^{\alpha u^2}-1}{|x|^a} \mathrm{d}x,\;\; u\in W^{1,2}_0(\Omega)
\end{equation}
is well-defined. Furthermore, it is bounded on $\Sigma=\{u\in W^{1,2}_0(\Omega)\;:\; \|u\|^2 = 1\}$ if and only if  $\alpha\le \alpha_a:=2\pi(2-a)$. For  $\alpha\le \alpha_a$, the existence of a (global) maximizer for $F_{\alpha}$ on $\Sigma$ is established in \cite{Roy} and consequently $(F_{\alpha})_{|_{\Sigma}}$ admits critical points on this range of $\alpha$.

The purpose of this paper is to establish the existence of local maximizers and, subsequently, of critical points of $F_{\alpha}$ restricted to $\Sigma$ beyond the critical threshold, namely, for $\alpha>\alpha_{a}$. Specifically, we prove the following result. 
\begin{theorem}[Local Maximizer]\label{S-max} There exists $\alpha^{*} > \alpha_a$ such that for any $\alpha \in (\alpha_a, \alpha^*)$, the  functional $F_{\alpha}$ in \eqref{SF-Moser} admits a local maximizer in $\Sigma$.
\end{theorem}
We point out that, when $a=0$, Theorem~\ref{S-max} recovers the result of Struwe \cite{Struwe}. Related results for the logarithmic Trudinger--Moser functional \cite{CR} have recently been obtained in \cite{Ruan}. To the best of our knowledge, the present work provides a first step toward the investigation of the existence of critical points beyond the critical constant for the singular Trudinger--Moser functional, in the same spirit as \cite{PMR,MM}.

\medskip

\textit{The outline of the paper is as follows:} In Section 2 we present some auxiliary results on the regularity and differentiability of the functional \(F_\alpha\). In Section 3 we study the critical case $\alpha_a$, proving the compactness of the set of maximizers and establishing a local variational structure. We then extend these arguments to the supercritical regime $\alpha > \alpha_a$ (close to $\alpha_a$) and complete the proof of the main result.

\section{Auxiliary results}
In this section, we establish some auxiliary results that will be used in the proof of Theorem~\ref{S-max}.
\begin{lemma}\label{Lema-Frechet} 
	For $\alpha>0$, let $F_{\alpha}: W^{1,2}_{0}(\Omega)\to \mathbb{R}$ be defined by \eqref{SF-Moser}. Then $F_{\alpha}$ is  continuously Fr\'{e}chet differentiable with
	$$ 
	F_\alpha^\prime(u)\cdot v = 2\alpha\int_{\Omega} \frac{ue^{\alpha u^2}v}{|x|^a}\mathrm{d}x,\quad  \forall v\in W^{1,2}_{0}(\Omega),
	$$
for any $u\in W^{1,2}_{0}(\Omega)$. Furthermore,  $F^{\prime}_{\alpha}: W^{1,2}_{0}(\Omega)  \to  W^{-1,2}(\Omega) $ is locally uniformly continuous and bounded.
	\end{lemma}
\begin{proof}
We will apply the Vitali convergence theorem. First, for $u, v \in W^{1,2}_{0}(\Omega)$ and $(s_m)$ any positive real sequence with $s_m\to 0$, we observe that
\begin{equation}\label{incremento}
	\frac{F_\alpha(u+s_mv)-F_\alpha(u)}{s_m} = \int_{\Omega}f_m(x)\mathrm{d}x, 
\end{equation}
where $f_m:\Omega\to\mathbb{R}$ is given by
\begin{equation*}
f_{m}(x) = \frac{1}{s_m}\int_{0}^{s_m} \frac{2\alpha v(u+tv) e^{\alpha(u+tv)^2}}{|x|^a} \mathrm{d}t.
\end{equation*}
We claim that $(f_m)$ converges strongly to $f$ in $L^1(\Omega)$ where $f:\Omega\to\mathbb{R}$ is defined by
\begin{equation*}
f(x) = \frac{2\alpha uv e^{\alpha u^2}}{|x|^a}.
\end{equation*}
In fact, we have $f_{m}(x)\to f(x)$ a.e. in $\Omega$. In addition, by Hölder's inequality and the singular Trudinger-Moser inequality \eqref{SF-Moser}, we get $f\in L^1(\Omega)$. In order to ensure the uniform integrability of $(f_m)$, it is sufficient to show that $(f_m)$ is a bounded sequence in $L^{q}(B)$ for some $q>1$. Since $a \in [0, 2)$, we can choose $q>1$ such that $aq < 2$. By Hölder's inequality with respect to $t$, we get
\begin{align*}
|f_m(x)|^{q} \leq \frac{1}{s_{m}}\int_{0}^{s_m} \frac{(2\alpha)^q |v|^q |u+tv|^q e^{q\alpha(u+t v)^2}}{|x|^{aq}} \mathrm{d}t,\quad \forall x \in \Omega.
\end{align*}
Since $\tau^q e^{\sigma \tau^2} \le c e^{c\tau^2}, \tau\ge 0$ for some $c>0$ depending only on $q$ and $\sigma$, using the elementary inequality $(a+b)^2 \le 2a^2 + 2b^2$,  for any $t\in [0, s_m]$ it follows 
\begin{equation}\label{estimate-JL}
\begin{aligned}
|u+t v|^q e^{q\alpha|u+tv|^2} &\leq c e^{c|u+tv|^2}\leq c e^{2cu^2} e^{2cs_m^2v^2}.
\end{aligned}
\end{equation}
Let $\gamma>1$ such that $aq\gamma<2$.  For $s_m$ sufficiently small, we have $2c s_m^2\gamma\|v\|^2 \le \alpha_a$. By applying the generalized Hölder's inequality to the product of $|v|^q$, $e^{2c|u|^2}$, and $e^{2c s_m^2|v|^2}/|x|^{aq}$, and using the singular Trudinger-Moser inequality \eqref{SF-Moser}, we obtain
\begin{equation}\label{unifboun}
\begin{aligned}
\int_{\Omega}|f_m|^q\mathrm{d}x &\le \frac{c (2\alpha)^q}{s_{m}}\int_{0}^{s_m}\int_{\Omega} |v|^{q} e^{2cu^2} \frac{e^{2cs_m^2v^2}}{|x|^{aq}} \mathrm{d}x\mathrm{d}t \\
&\le  c (2\alpha)^q\Big(\int_{\Omega}|v|^{\frac{2q\gamma}{\gamma-1}}\mathrm{d}x\Big)^{\frac{\gamma-1}{2\gamma}}\Big(\int_{\Omega}e^{\frac{4c\gamma}{\gamma-1}u^2}\mathrm{d}x\Big)^{\frac{\gamma-1}{2\gamma}}\Big(\int_{\Omega}\frac{e^{2\gamma cs_m^2v^2}}{|x|^{aq\gamma}} \mathrm{d}x\Big)^{\frac{\gamma-1}{\gamma}}\\
&\le C,
\end{aligned}
\end{equation}
for some $C>0$ independent of $m$.  Hence, from Vitali's convergence theorem, our claim holds. By dividing \eqref{incremento} by $s_m$ and letting $m\to\infty$, we obtain the Gâteaux derivative
$$
F_\alpha^\prime(u)\cdot v = 2\alpha\int_{\Omega} \frac{ u v e^{\alpha u^2}}{|x|^a}\mathrm{d}x.
$$
It remains to prove that $F^{\prime}_{\alpha}: W^{1,2}_{0}(\Omega) \to W^{-1,2}(\Omega)$ is locally uniformly continuous and bounded. First, we show local boundedness. Let $u \in W^{1,2}_{0}(\Omega)$ be arbitrarily fixed and let $U(u;\varepsilon)=\{w\in W^{1,2}_{0}(\Omega) : \|w-u\|<\varepsilon\}$. For $w\in U(u;\varepsilon)$, we set $h=(w-u)/\varepsilon$. From  Hölder's inequality and arguing as in \eqref{estimate-JL} and \eqref{unifboun}, we have
\begin{equation}\label{F-bounded}
	\begin{aligned}
		|F^\prime_{\alpha}(w)\cdot v|
		&\leq \int_{\Omega} \frac{2\alpha|v||u+\varepsilon h| e^{\alpha|u+\varepsilon h|^2}}{|x|^a}\,\mathrm{d}x\\
		&\leq C \int_\Omega |v| e^{c u^2} \frac{e^{c\varepsilon^2 h^2}}{|x|^a} \,\mathrm{d}x\\
		&\leq C 
		\Big(\int_\Omega |v|^{\frac{2\gamma}{\gamma-1}}\,\mathrm{d}x\Big)^{\frac{\gamma-1}{2\gamma}}
		\Big(\int_\Omega e^{\frac{2c\gamma}{\gamma-1}u^2}\,\mathrm{d}x\Big)^{\frac{\gamma-1}{2\gamma}}
		\Big(\int_\Omega \frac{e^{2c\gamma \varepsilon^2 h^2}}{|x|^{a\gamma}}\,\mathrm{d}x\Big)^{\frac{1}{\gamma}},
	\end{aligned}
\end{equation}
for appropriate exponents $\gamma > 1$ with $a \gamma < 2$. By choosing $\varepsilon>0$ small enough such that $2c \varepsilon^2 \gamma \leq \alpha_a$, the last integral is bounded by the singular Trudinger-Moser inequality \eqref{SF-Moser}. Thus, taking the supremum over $\|v\| = 1$ and using the continuous embedding  $W^{1,2}_{0}(\Omega)\hookrightarrow L^{p}(\Omega)$ for all $p\ge 1$, we can see from \eqref{F-bounded} that $F^\prime_\alpha$ is locally bounded. Now, for any $v\in W^{1,2}_{0}(\Omega)$ with $\|v\|\le 1$, we evaluate the difference
\begin{equation}\label{diferençaI}
\begin{aligned}
|F_\alpha^\prime(u+\varepsilon h)\cdot v-F_\alpha^\prime(u)\cdot v| &= \Big| \int_\Omega\frac{2\alpha v}{|x|^a} \int_{0}^{\varepsilon} \frac{d}{dt}\Big((u+th)e^{\alpha|u+th|^2}\Big)\mathrm{d}t\mathrm{d}x\Big|\\
& \le \int_{0}^{\varepsilon}\int_\Omega\frac{2\alpha|v||h| e^{\alpha|u+th|^2}}{|x|^a}\mathrm{d}t\mathrm{d}x\\
& + \int_{0}^{\varepsilon}\int_\Omega \frac{4\alpha^2|v||h||u+th|^2 e^{\alpha|u+th|^2}}{|x|^a}\mathrm{d}t\mathrm{d}x.
\end{aligned}
\end{equation}
Let us estimate the first integral in \eqref{diferençaI}. Similarly to \eqref{estimate-JL}, for all $t\in [0, \varepsilon]$ we obtain
\begin{equation}\label{I-part1}
\begin{aligned}
\int_\Omega \frac{|v||h| e^{\alpha|u+th|^2}}{|x|^a}\mathrm{d}x &\leq \int_\Omega |v||h| e^{2\alpha u^2} \frac{e^{2\alpha\varepsilon^2h^2}}{|x|^a}\mathrm{d}x\\
&\leq \|v\|_{L^{\frac{3\gamma}{\gamma-1}}} \|h\|_{L^{\frac{3\gamma}{\gamma-1}}}\left(\int_{\Omega}e^{\frac{6\alpha}{\gamma-1} u^2}\mathrm{d}x\right)^{\frac{\gamma-1}{3\gamma}}
 \left(\int_\Omega \frac{e^{2\varepsilon^2\alpha \gamma h^2}}{|x|^{a\gamma}}\mathrm{d}x\right)^{\frac{1}{\gamma}}\\
&\le C\left(\int_{\Omega}e^{\frac{6\alpha}{\gamma-1} u^2}\mathrm{d}x\right)^{\frac{\gamma-1}{3\gamma}},
\end{aligned}
\end{equation}
for $\gamma>1$ such that $\gamma a<2$ and $\varepsilon>0$ is small such that $2\varepsilon^2\gamma\alpha\le \alpha_{a}$, we obtain that the $h$ exponential term is bounded. The second integral in \eqref{diferençaI} can be estimated analogously. Combining these estimates and integrating over $t \in [0, \varepsilon]$, for $\varepsilon>0$ small enough \eqref{diferençaI} yields
\begin{equation*}
\|F_\alpha^\prime(u+\varepsilon h)-F_\alpha^\prime(u)\|_{\ast} \leq C \varepsilon,
\end{equation*}
for some constant $C>0$ depending only on $\alpha, a,$ and $u$. Therefore, $F^\prime_\alpha$ is locally uniformly continuous, which, combined with the existence of the Gâteaux derivative, ensures that $F_{\alpha}$ is continuously Fréchet differentiable.\end{proof}
	\begin{lemma}\label{regularity}
		If $u \in W^{1,2}_{0}(\Omega) $ weakly solves the equation
        \begin{equation}\label{Teste}
    \begin{cases}
        -\Delta u= \lambda \dfrac{ue^{\alpha_a |u|^{2}}}{|x|^{a}},  & \text{in } \Omega, \\
        u = 0, & \text{on } \partial \Omega
    \end{cases}
\end{equation}
where $\lambda\in\mathbb{R}$,
that is,
\begin{equation}\label{weak-solve-Def}
			\int_\Omega \nabla u \nabla v \, \mathrm{d}x = \lambda  \int_\Omega \dfrac{ue^{\alpha_a|u|^{2}}}{|x|^a} v \, \mathrm{d}x, \quad\mbox{for all}\;\; v\in   W^{1,2}_{0}(\Omega).
		\end{equation}
Then $u \in L^\infty(\Omega)$.	
	\end{lemma}
	
	\begin{proof}
		Let $u\in W^{1,2}_{0}(\Omega)$ be a weak solution of \eqref{Teste}. Then
		\[
		-\Delta u = f(x)\quad \text{in } \Omega,
		\]
		where
		\[
		f(x)=\lambda \frac{u e^{\alpha_a |u|^2}}{|x|^a}.
		\]
		Since $u\in W^{1,2}_{0}(\Omega)$, the Sobolev embedding theorem yields
		\[
		u\in L^q(\Omega) \quad \text{for every } 1\le q<\infty .
		\]
		
		Let $p>1$ be sufficiently close to $1$ such that $ap<2$. By Hölder's inequality, for $r,r'>1$ with $1/r+1/r'=1$, we obtain
		\[
		\int_\Omega |f(x)|^p \mathrm{d}x 
		\le |\lambda|^p 
		\left(\int_\Omega |u|^{r'p} \mathrm{d}x\right)^{1/r'}
		\left(\int_\Omega \frac{e^{rp\alpha_a |u|^2}}{|x|^{rap}} \mathrm{d}x\right)^{1/r}.
		\]
		Since $u\in L^q(\Omega)$ for every $1\le q<\infty$, the first factor is finite. 
		On the other hand, choosing $r>1$ sufficiently close to $1$ so that $rap<2$, the singular Trudinger--Moser inequality \eqref{SF-Moser} ensures that
		\[
		\int_\Omega \frac{e^{rp\alpha_a |u|^2}}{|x|^{rap}}\,\mathrm{d}x < \infty.
		\]
		Consequently,
		$
		f\in L^p(\Omega)
		$
		for some $p>1$.  Now we can apply the standard elliptic regularity theory for the Dirichlet problem
		\[
		-\Delta u = f \quad \text{in } \Omega, \qquad u=0 \text{ on }\partial \Omega,
		\]
		with $f\in L^p(\Omega)$, which yields
		$
		u\in W^{2,p}(\Omega).
		$
		Since $p>1$ and the dimension is two, the Sobolev embedding theorem implies
		\[
		W^{2,p}(\Omega) \hookrightarrow C^{0,\theta}(\overline{\Omega})
		\]
		for some $\theta\in(0,1)$. In particular,
	$
		u\in L^\infty(\Omega)
		$
		and this completes the proof.
	\end{proof}

\section{Maximizer beyond the critical constant}
The proof of the Theorem~\ref{S-max} will be given in this section through a series of lemmas. Firstly, 
let us define
\begin{equation}\label{Sigma-sup}
   T(a,\Omega)=:\sup_{\Sigma}F_{\alpha_a}(u).
\end{equation}
\begin{lemma}\label{lemacompacto}
		The nonempty maximizer set $$
		K_{\alpha_a}=\{u \in \Sigma\,:\, F_{\alpha_a}(u)=T(a,\Omega)\}
		$$ is  compact.
	\end{lemma}
    \begin{proof}
      First, from \cite[Theorem~1]{Roy} we have $	K_{\alpha_a}\not=\emptyset$.  Let $(u_m)$ be a sequence in $K_{\alpha_{a}}$, that is, $\|u_m\|=1$ and $ F_{\alpha_a}(u_m)=T(a,\Omega)$. Thus, $
      (u_m)$ is a maximizing sequence, that is,
      \begin{equation}\label{u_m-maxi}
      	\lim_m F_{\alpha_a}(u_m)=T(a,\Omega).
      \end{equation}
      Up to a subsequence, we may assume that 
      $u_m \rightharpoonup u$ in $W^{1,2}_{0}(\Omega)$.
      Hence, by the concentration--compactness alternative \cite[Theorem~6]{Roy}, either
      \begin{itemize}
      	\item [$(i)$] $(u_m)$ concentrates at a point $x_0\in\overline{\Omega}$ or
      	\item [$(ii)$]  $\displaystyle\lim_{m\to\infty}F_{\alpha_a}(u_m)=F_{\alpha_a}(u)$.
      \end{itemize}
      If $(i)$ occurs for some $x_0\in\overline{\Omega}\setminus\left\{0\right\}$, \cite[Proposition~7]{Roy}  ensures that $\lim_m F_{\alpha_a}(u_m)\to F_{\alpha_a}(0)=0$ which contradicts \eqref{u_m-maxi}.  Therefore,  the only possibility for case $(i)$ to occur is with $x_0=0$. On the other hand, in this case, \cite[Theorem~16, Theorem~21, Proposition~22]{Roy} implies that 
      $$
      	\lim_{m\to\infty} F_{\alpha_a}(u_m)<T(a,\Omega),
      $$
      which contradicts \eqref{u_m-maxi}.  Then, we must have the alternative  $(ii)$. Hence, from \eqref{u_m-maxi} and  $(ii)$, we obtain $u\not\equiv0$ and $T(a,\Omega)=F_{\alpha_a}(u)$. From the weak lower semicontinuity of the norm, we also have $$0<\|u\|\le\liminf_{m\to\infty} \|u_m\|=1.$$
      Hence,  
      $T(a,\Omega)=F_{\alpha_a}(u)\le F_{\alpha_a}\Big(\frac{u}{\|u\|}\Big)\le T(a,\Omega).$
      It follows that 
      \begin{equation*}
      	\int_{\Omega}\frac{1}{|x|^a}\Big(e^{\alpha_a\left(\frac{|u|}{\|u\|}\right)^2}-e^{\alpha_a|u|^2}\Big)\mathrm{d}x=0.
      \end{equation*}
      Since $0<\|u\|\le 1$, the integrand above is nonnegative and, consequently, $\|u\|=1$. Thus, we obtain $u\in K_{\alpha_{a}}$ and $u_m\to u$ strongly in $W^{1,2}_{0}(\Omega) $.
    \end{proof}
In view of Lemma~\ref{lemacompacto}, for every $\varepsilon>0$, we set
$$
N_\varepsilon=\{u \in \Sigma:\; \mathrm{dist}(u, K_{\alpha_a})<\varepsilon\},
$$
where $\mathrm{dist}(u, K_{\alpha_a})=\inf\{\|u-v\|\,:\,v\in K_{\alpha_a}\}$. 
The family $(N_{\varepsilon})_{\varepsilon>0}$  constitutes a neighborhood basis for $K_{\alpha_a}$ in $\Sigma$, i.e.
\begin{enumerate}
	\item $ K_{\alpha_a}\subset N_\varepsilon$, for all $\varepsilon>0$;
	\item If $U$ is an open set in $\Sigma$ such that $K_{\alpha_a}\subset U$, then $ N_{\varepsilon}\subset U$ for some $\varepsilon>0$.
\end{enumerate}
\begin{lemma} For every $\varepsilon>0$ it holds
\begin{equation}\label{s-compare}
	T(a,\Omega)= \sup_{u\in N_\varepsilon}F_{\alpha_a}(u).
\end{equation}
\end{lemma}
\begin{proof}
First, we observe that   
$$
\sup_{u\in N_\varepsilon}F_{\alpha_a}(u)\le \sup_{u\in\Sigma}F_{\alpha_a}(u)= T(a,\Omega),  \quad\mbox{for all}\;\; \varepsilon>0.
$$
On the other hand, using that $T(a,\Omega)$ is attained, see \cite{Roy}, we get 
 $K_{\alpha_a}\not=\emptyset$ which together with $K_{\alpha_a}\subset N_\varepsilon$ yields  
\begin{equation}\label{s-compare}
	T(a,\Omega)=F_{\alpha_a}(u)=\sup_{u\in\Sigma}F_{\alpha_a}(u)\leq \sup_{u\in N_\varepsilon}F_{\alpha_a}(u).
\end{equation}
Thus, we obtain \eqref{s-compare} and this completes the proof.
\end{proof}

According to \eqref{s-compare}, the critical maximal value is preserved in a small neighborhood of the set of maximizers. Furthermore, by the inclusion $N_{2\varepsilon}\backslash N_\varepsilon\subset\Sigma$, one has
$$
\sup_{u\in N_{2\varepsilon}\backslash N_\varepsilon} F_{\alpha_a}(u)\le \sup_{u\in \Sigma} F_{\alpha_a}(u) =T(a,\Omega).
$$
We next prove that this inequality is, in fact, strict.
\begin{figure}[h!]
	\centering
	\begin{tikzpicture}[scale=1.5]
		
		\fill[blue!20]
		(0,0.5) .. controls (2,1.5) .. (4,0.5)
		-- (4,1) .. controls (2,2) .. (0,1)
		-- cycle;
		
		\fill[blue!20]
		(0,-0.5) .. controls (2,0.5) .. (4,-0.5)
		-- (4,-1) .. controls (2,0) .. (0,-1)
		-- cycle;
		
		\draw[thick,blue] 
		(0,0) 
		.. controls (2,1) .. 
		node[pos=-0.04] {$K_{\alpha_a}$}
		(4,0);
		
		\draw[thick] (0,0.5) .. controls (2,1.5) .. (4,0.5);
		\draw[dashed] (0,1) .. controls (2,2) .. (4,1);
		
		\draw[thick] (0,-0.5) .. controls (2,0.5) .. (4,-0.5);
		\draw[dashed] (0,-1) .. controls (2,0) .. (4,-1);
		
		\draw[decorate,decoration={brace}] (4.1,0.5) -- (4.1,0);
		\node[right] at (4.1,0.25) {$\varepsilon$};
		
		\draw[decorate,decoration={brace}] (4.5,0.9) -- (4.5,0);
		\node[right] at (4.5,0.5) {$2\varepsilon$};
		
	\end{tikzpicture}
	\caption{$N_{2\varepsilon}\setminus N_\varepsilon$.}
	\label{fig:vizinhanca-kmu}
\end{figure}
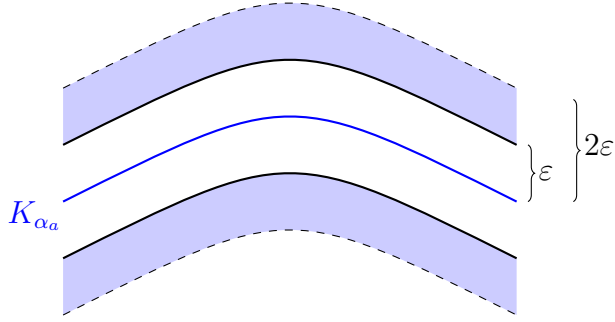

\begin{lemma}\label{supercriticocritico} For every $\varepsilon>0$, the following strict inequality holds:
	\begin{equation}\label{maximorestrito1}
		\sup_{u\in N_{2\varepsilon}\backslash N_\varepsilon} F_{\alpha_a}(u)<T(a,\Omega)=\sup_{u\in N_\varepsilon}F_{\alpha_a}(u).
	\end{equation}
\end{lemma}
\begin{proof}
    By contradiction, assume that there exists $\varepsilon_0>0$ such that  
    \begin{equation}\nonumber
    	\sup_{u\in N_{2\varepsilon_0}\backslash N_{\varepsilon_0}} F_{\alpha_a}(u)=T(a,\Omega).
    \end{equation}
   Let $(u_m)\subset N_{2\varepsilon_0}\setminus N_{\varepsilon_0}$ be a sequence such that $ F_{\alpha_a}(u_m)\to T(a,\Omega)$. Up to a subsequence we may assume that $ u_m \rightharpoonup u$ in $ W^{1,2}_{0}(\Omega)$. Arguing as in the proof of Lemma~\ref{lemacompacto} and applying   \cite[Theorem~6, Proposition~7, Theorem~16, Theorem~21, Proposition~22]{Roy}, we are able to obtain that
    \begin{equation}\label{J(u)=b}
    	\|u\|=1
    	\quad \text{and} \quad
    	F_{\alpha_a}(u)=T(a,\Omega).
    \end{equation}
    As a consequence, $u_m \to u$ strongly in $W^{1,2}_{0}(\Omega)$ and  $u\in K_{\alpha_a}$. Thus, 
    \[
   0= \mathrm{dist}(u,K_{\alpha_a})
    =\lim_{m\to\infty}\mathrm{dist}(u_m,K_{\alpha_a})
    \ge \varepsilon_0>0,
    \]
    which is a contradiction and this completes the proof.
\end{proof}
Next, we show that the strict inequality \eqref{maximorestrito1} still holds in the supercritical regime $\alpha>\alpha_a$, at least for $\alpha$ close to $\alpha_a$.
\begin{lemma}\label{supercriticoalém}
		There exists $\alpha^{\sharp}> \alpha_a$ such that for every $\alpha \in [\alpha_a, \alpha^{\sharp})$, we have
		\begin{equation}\label{maximorestritomu}
			\sup_{N_{2\varepsilon}\backslash N_\varepsilon} F_{\alpha}(u)<F^\sharp_{\alpha}:=\sup_{u\in N_\varepsilon}F_{\alpha}(u),
		\end{equation}
	for all $ \varepsilon>0$ small enough. 	
	\end{lemma}
    \begin{proof}
    By Lemma~\ref{Lema-Frechet},  we know that the functional  $F_{\alpha}$ is locally uniformly  continuous.  For each $v \in K_{\alpha_a}$, let $U(v;\varepsilon_v)=\{u\in W^{1,2}_{0}(\Omega)\;: \; \|u-v\|<\varepsilon_v\}$ be a neighborhood of $v$ such that  $F_{\alpha}$ is uniformly continuous. Since by Lemma~\ref{lemacompacto} $K_{\alpha_a}$ is compact, there exist  $v_1,\cdots, v_m$ such that  
    $$
    K_{\alpha_a} \subset \bigcup_{i=1}^m U(v_i; \frac{\varepsilon_i}{2}).
    $$
    Set  $U= \bigcup_{i=1}^m U(v_i; \frac{\varepsilon_i}{2}) $. Using that  $F_{\alpha}$ is uniformly continuous on each $U(v_i; \varepsilon_i)$, given $\varepsilon_0>0$ there exists $\sigma>0$ such that
    \begin{equation}\label{uniform-IH}
    	u, v\in U(v_i; \varepsilon_i), \; \|u-v\|<\sigma\;\;\mbox{implies}\;\; |F_{\alpha}(u)-F_{\alpha}(v)|<\varepsilon_0,
    \end{equation}
     for all $i=1,\cdots, m$.
    Set $\varepsilon_M=\max\{\varepsilon_1,\cdots, \varepsilon_m\}$. For $\alpha, \beta>0$, let us choose $\delta>0$ such that 
    \begin{equation}\label{delta-choice}
    	|\alpha-\beta|<\delta\;\;\mbox{implies}\;\;  \Big|1-\Big(\frac{\beta}{\alpha}\Big)^{\frac{1}{2}}\Big|\left(1+\frac{\varepsilon_M}{2}\right)<\min\left\{\frac{\varepsilon_1}{2}, \cdots, \frac{\varepsilon_m}{2},\sigma\right\}.
    \end{equation}
    Now, if $u \in U$  then $u\in U(v_{i}; \frac{\varepsilon_{i}}{2})$ for some $i \in \{1, \cdots, m\}$. From our choice of $\delta$ in \eqref{delta-choice}, for $|\mu-\nu|<\delta$  we obtain 
    \begin{equation*}
    	\begin{aligned}
    		\Big\|u - \Big(\frac{\beta}{\alpha}\Big)^{\frac{1}{2}}u\Big\|=
    		\Big|1-\Big(\frac{\beta}{\alpha}\Big)^{\frac{1}{2}}\Big|\|u\|
    		\leq&  \Big|1-\Big(\frac{\beta}{\alpha}\Big)^{\frac{1}{2}}\Big|\left(\frac{\varepsilon_i}{2}+\|v_i\|\right)\\
    		\leq& \Big|1-\Big(\frac{\beta}{\alpha}\Big)^{\frac{1}{2}}\Big|\left(\frac{\varepsilon_M}{2}+1\right)<\sigma 
    	\end{aligned}
    \end{equation*}
    and 
    \begin{equation*}
    	\begin{aligned}
    		\Big\|v_i - \Big(\frac{\beta}{\alpha}\Big)^{\frac{1}{2}}u\Big\|& \leq\|v_i-u\|+\Big\|u - \Big(\frac{\beta}{\alpha}\Big)^{\frac{1}{2}}u \Big\|\\
    		& \leq \frac{\varepsilon_i}{2}+ \Big|1-\Big(\frac{\beta}{\alpha}\Big)^{\frac{1}{2}}\Big|\left(\frac{\varepsilon_M}{2}+1\right)\\&<\varepsilon_i. 
    	\end{aligned}
    \end{equation*}
    Thus, we conclude that
    $$
    \Big(\frac{\beta}{\alpha}\Big)^{\frac{1}{2}}u\in U(v_{i}; \varepsilon_{i})\;\;\mbox{and}\;\; \Big\|u-\Big(\frac{\beta}{\alpha}\Big)^{\frac{1}{2}}u\Big\|<\sigma.
    $$ 
    From \eqref{uniform-IH}, it follows that 
    \begin{equation}\label{uniform-Iconti}
    	\begin{aligned}
    		|\alpha-\beta|<\delta \;\;\mbox{implies}\;\; |F_{\alpha}(u)-F_{\beta}(u)|=\Big|F_{\alpha}(u)-F_{\alpha}\Big(\Big(\frac{\beta}{\alpha}\Big)^{\frac{1}{2}}u\Big)\Big|<\varepsilon_0.
    	\end{aligned}
    \end{equation}
    Consequently, the map $\alpha\mapsto F_{\alpha}(u)$ is continuous, uniformly with respect to  $u\in U$. Now, let us take $\varepsilon>0$ small such that $N_{2\varepsilon}\subset U$. From \eqref{maximorestrito1}, we are able to choose $\varepsilon_0>0$ such that 
    \begin{equation}\label{maximorestrito12}
    	\sup_{u\in N_{2\varepsilon}\backslash N_\varepsilon} F_{\alpha_a}(u)<\sup_{u\in N_\varepsilon}F_{\alpha_a}(u)-2\varepsilon_0.
    \end{equation}
    For $\beta=\alpha_{a}$ and such choice of $\varepsilon_0$, take $\delta>0$ such that \eqref{uniform-Iconti} holds. Hence, from  \eqref{uniform-Iconti} and \eqref{maximorestrito12}, for all $\alpha\in [\alpha_{a},\alpha_{a}+\delta)$ we obtain
    \begin{equation*}
    	\begin{aligned}
    		\sup_{u\in N_{2\varepsilon}\backslash N_\varepsilon} F_{\alpha}(u)\le \varepsilon_0+ \sup_{u\in N_{2\varepsilon}\backslash N_\varepsilon} F _{\alpha_{a}}(u)
    		&< \sup_{u\in N_\varepsilon}F_{\alpha_a}(u)-\varepsilon_0
    		&\le \sup_{u\in N_\varepsilon}F_{\alpha}(u),
    	\end{aligned} 
    \end{equation*}
    which completes the proof.
\end{proof}
    \begin{proof}[Proof of Theorem~\ref{S-max} ] Let $\alpha^{\sharp}>\alpha_a$ be given by Lemma~\ref{supercriticoalém} and $F^\sharp_{\alpha}$ as in \eqref{maximorestritomu}.
     Now, for  $\alpha\in[\alpha_{a}, \alpha^{\sharp})$, let $(u_m) \subset N_\varepsilon$ be a maximizing sequence for $F^\sharp_{\alpha}$, i.e.   $F_{\alpha}(u_m) \to F^\sharp_{\alpha}$ as $m\to\infty$, which we can assume that $u_m \rightharpoonup u$ weakly in $W^{1,2}_{0}(\Omega)$ and $u_m(x)\to u(x)$ a.e in $\Omega$. For each $m\in \mathbb{N}$, let $(v_m) \subset K_{\alpha_a}$ satisfying $\|u_m-v_m\|\le \varepsilon$. By compactness, up to  a subsequence, we have $v_m \to v$ strongly in $W^{1,2}_{0}(\Omega)$ for some $v\in K_{\alpha_{a}}$. In particular,  $v$ is a solution of $\eqref{Teste}$ and the Lemma~\ref{regularity} yields $v \in L^\infty(\Omega)$.  Set $w_m=u_m-v_m+v$ and  note that $w_m-u_m \to 0$ strongly in $W^{1,2}_{0}(\Omega)$. For $a,b>0$, Young's inequality yields
     $$
     (a+b)^2\leq (1+\delta)a^2+(1+\frac{1}{\delta})b^2, \quad\mbox{for every}\, \delta>0.
     $$
Picking $\delta=\sqrt{2}-1$, 
    it follows that 
    \begin{equation}\label{e-convex}
    |w_m|^{2} \le (|v|+|u_m-v_m|)^{2}
    		\le \sqrt{2}\|v\|^2_{L^{\infty}}+\frac{\sqrt{2}}{\sqrt{2}-1}\varepsilon^2\Big|\frac{u_m-v_m}{\|u_m-v_m\|}\Big|^{2}.
  \end{equation}
  Fix $q>1$ (close to $1$) with $aq<2$.  Then, choose $\varepsilon>0$ small enough to ensure that $\alpha_{\epsilon,q}:=q\alpha\frac{\sqrt{2}}{\sqrt{2}-1} \varepsilon^2\le \alpha_{aq}=2\pi(2-aq)$ for all $\alpha\in [\alpha_{a}, \alpha^{\sharp})$.  From  \eqref{e-convex}  we obtain
    \begin{align*}
    \int_{\Omega}\frac{e^{q\alpha |w_m|^2}}{|x|^{aq}}{\rm d}x& \le  e^{q\alpha\sqrt{2}\|v\|^2_{L^{\infty}}}\int_{\Omega}\frac{e^{\alpha_{\epsilon,q}\big|\frac{u_m-v_m}{\|u_m-v_m\|}\big|^{2}}}{|x|^{aq}}{\rm d}x\\
&\le \ e^{q\alpha\sqrt{2}\|v\|^2_{L^{\infty}}}\int_{\Omega}\frac{e^{\alpha_{aq}\big|\frac{u_m-v_m}{\|u_m-v_m\|}\big|^{2}}}{|x|^{aq}}{\rm d}x.
    \end{align*}
   It follows from the singular Trudinger-Moser inequality in \cite{AS}  that  $(\exp(\alpha|w_m|^{2})-1)/|x|^a$ is uniformly bounded in $L^q(\Omega)$ for some $q>1$.
   
   Thus, the Vitali convergence theorem implies $F_{\alpha}(w_m) \to F_{\alpha}(u)$ and the local uniform continuity of $F_\alpha$ yields
   $F_{\alpha}(w_m)-F_{\alpha}(u_m) \to 0$. Hence 
    \begin{align*}
    F^{\sharp}_\alpha=\lim_{m\to\infty}[(F_{\alpha}(u_m)-F_{\alpha}(w_m))+F_{\alpha}(w_m)]=	F_{\alpha}(u).
    \end{align*}
     In addition, by the lower semi-continuity of the norm  we have $\|u-v\| \leq \varepsilon$ and using $\|v\|= 1$, we obtain
    \begin{equation*}
    	\begin{aligned}
    		\left\|v-\frac{u}{\|u\|}\right\|&\leq
    		\|v-u\|+|1-\|u\||\\
    		&=\|v-u\|+|\|v\|-\|u\||\\
    		&\leq 2 \|v-u\| \leq 2 \varepsilon. 
    	\end{aligned}
    \end{equation*}
    This implies that $\frac{u}{\|u\|} \in \overline{N}_{2\varepsilon}$ and from \eqref{maximorestritomu}, we obtain  
    $F_{\alpha}\left(\frac{u}{\|u\|}\right) \leq F^{\sharp}_{\alpha} = F_{\alpha}(u).$
    It follows that $\|u\| = 1$. Consequently, $u \in \Sigma$ which implies $u \in \overline{N}_\varepsilon$. In fact, by \eqref{maximorestritomu}, we must have  $u \in N_{\varepsilon}$, and hence  $F^\sharp_{\alpha}$ is attained in $N_\varepsilon$. 
\end{proof}
\bibliographystyle{abbrv}
\bibliography{reference}

\end{document}